\documentclass[10pt,leqno]{article}

\usepackage{amsmath,amssymb,amsthm,mathrsfs,dsfont}

\usepackage[margin=3cm]{geometry} 

\usepackage{titlesec,hyperref}

\usepackage{color}

\usepackage{fancyhdr}
\titleformat{\subsection}{\it}{\thesubsection.\enspace}{1pt}{}

\newtheorem{theo}{Theorem}[section]
\newtheorem{lemm}[theo]{Lemma}

\newtheorem{coro}[theo]{Corollary}
\newtheorem{prop}[theo]{Proposition}
\newtheorem{rema}[theo]{Remark}
\numberwithin{equation}{section}

\allowdisplaybreaks 

\newcommand\ep{{\varepsilon}} 

\begin{document}
\title{Large time behavior to the FENE dumbbell model of polymeric flows near equilibrium
\hspace{-4mm}
}

\author{ Zhaonan $\mbox{Luo}^1$ \footnote{email: 1411919168@qq.com},\quad
Wei $\mbox{Luo}^1$\footnote{E-mail:  luowei23@mail2.sysu.edu.cn} \quad and\quad
 Zhaoyang $\mbox{Yin}^{1,2}$\footnote{E-mail: mcsyzy@mail.sysu.edu.cn}\\
 $^1\mbox{Department}$ of Mathematics,
Sun Yat-sen University, Guangzhou 510275, China\\
$^2\mbox{Faculty}$ of Information Technology,\\ Macau University of Science and Technology, Macau, China}

\date{}
\maketitle
\hrule

\begin{abstract}
In this paper we mainly study large time behavior for the strong solutions of the finite extensible nonlinear elastic (FENE) dumbbell model. There is a lot results about the $L^2$ decay rate of the co-rotation model. In this paper, we consider the general case. We prove that the optimal $L^2$ decay rate of the velocity is $(1+t)^{-\frac{d}{4}}$ with $d\geq 2$. This result improves the previous result in \cite{Luo-Yin}.\\

\vspace*{5pt}

\vspace*{5pt}
\end{abstract}

\vspace*{10pt}

\tableofcontents

\section{Introduction}
In this paper we investigate the large time behavior for the finite extensible nonlinear elastic (FENE) dumbbell model \cite{Bird1977,Doi1988}:
\begin{align}\label{eq0}
\left\{
\begin{array}{ll}
u_t+(u\cdot\nabla)u-\Delta u+\nabla{P}=div~\tau,  ~~~~~~~div~u=0,\\[1ex]
\psi_t+(u\cdot\nabla)\psi=div_{R}[-\sigma(u)\cdot{R}\psi+\beta\nabla_{R}\psi+\nabla_{R}\mathcal{U}\psi],  \\[1ex]
\tau_{ij}=\int_{B}(R_{i}\nabla_{R_j}\mathcal{U})\psi dR, \\[1ex]
u|_{t=0}=u_0,~ \psi|_{t=0}=\psi_0, \\[1ex]
(\beta\nabla_{R}\psi+\nabla_{R}\mathcal{U}\psi)\cdot{n}=0 ~~~~ \text{on} ~~~~ \partial B(0,R_0) .\\[1ex]
\end{array}
\right.
\end{align}
In \ref{eq0}~~$u(t,x)$ stands for the velocity of the polymeric liquid. The polymer particles are described by the distribution function $\psi(t,x,R)$. Here the polymer elongation $R$ is bounded in ball $ B=B(0,R_{0})$ which means that the extensibility of the polymers is finite and $x\in\mathbb{R}^n$.
$\tau$ is an extra-stress tensor which generated by the polymer particles effect and $P$ is the pressure. Moreover the potential $\mathcal{U}(R)=-k\log(1-(\frac{|R|}{|R_{0}|})^{2})$ for some $k>0$. $\sigma(u)$ is the drag term. In general, $\sigma(u)=\nabla u$. For the co-rotation case, $\sigma(u)=\frac{\nabla u-(\nabla u)^{T}}{2}$.

In the paper we will take $\sigma(u)=\nabla u$, $\beta=1$ and $R_{0}=1$.
Notice that $(u,\psi)$ with $u=0$ and $$\psi_{\infty}(R)=\frac{e^{-\mathcal{U}(R)}}{\int_{B}e^{-\mathcal{U}(R)}dR}=\frac{(1-|R|^2)^k}{\int_{B}(1-|R|^2)^kdR},$$
is a trivial solution of \ref{eq0}.
Let $\tilde{\psi}=\psi-\psi_\infty$, by a simple calculation, we can rewrite \ref{eq0} for the following system:
\begin{align}\label{eq1}
\left\{
\begin{array}{ll}
u_t+(u\cdot\nabla)u-\Delta u+\nabla{P}=div~\tau,  ~~~~~~~div~u=0,\\[1ex]
\tilde{\psi}_t+u\cdot\nabla\tilde{\psi}+\mathcal{L}(\tilde{\psi})=div_{R}(-\nabla u\cdot{R}\tilde{\psi})+\psi_{\infty}\nabla u\cdot{R}\nabla_{R}\mathcal{U}, \\[1ex]
\tau_{ij}=\int_{B}(R_{i}\nabla_{R_j}\mathcal{U})\tilde{\psi} dR, \\[1ex]
u|_{t=0}=u_0,~\tilde{\psi}|_{t=0}=\psi_0-\psi_\infty, \\[1ex]
\psi_{\infty}\nabla_{R}\frac{\tilde{\psi}}{\psi_{\infty}}\cdot{n}=0 ~~~~ \text{on} ~~~~ \partial B(0,1) ,\\[1ex]
\end{array}
\right.
\end{align}
where $\mathcal{L}(\tilde{\psi})=-div_{R}(\psi_{\infty}\nabla_{R}\frac{\tilde{\psi}}{\psi_{\infty}})$. \\
{\bf Remark.} As in the reference \cite{Masmoudi2013}, one can deduce that $\psi=0$ on $\partial B(0,1)$.

There are a lot of mathematical results about the dumbbell model. M. Renardy \cite{Renardy} established the local well-posedness in Sobolev spaces with potential $\mathcal{U}(R)=(1-|R|^2)^{1-\sigma}$ for $\sigma>1$. Later, B. Jourdain, T. Leli\`{e}vre, and
C. Le Bris \cite{Jourdain} proved local existence of a stochastic differential equation with potential $\mathcal{U}(R)=-k\log(1-|R|^{2})$ in the case $k>3$ for a Couette flow. H. Zhang and P. Zhang \cite{Zhang-H} proved local well-posedness of \eqref{eq1} with $d=3$ in weighted Sobolev spaces. For the co-rotation case, F. Lin, P. Zhang, and Z. Zhang \cite{F.Lin} obtain a global existence results with $d=2$ and $k > 6$. If the initial data is perturbation around equilibrium, N. Masmoudi \cite{Masmoudi2008} proved global well-posedness of \eqref{eq1} for $k>0$. In the co-rotation case with $d=2$, he \cite{Masmoudi2008} obtained a global result for $k>0$ without any small conditions. W. Luo and Z. Yin improved the result to Besov spaces \cite{Luo-Yin-NA}. In the co-rotation case, A. V. Busuioc, I. S. Ciuperca, D. Iftimie and L. I. Palade \cite{Busuioc} obtain a global existence result with only the small condition on $\psi_0$. The global existence of weak solutions in $L^2$ was proved recently by N. Masmoudi \cite{Masmoudi2013} under some entropy conditions.

Recently, M. Schonbek \cite{Schonbek} studied the $L^2$ decay of the velocity for the co-rotation
FENE dumbbell model, and obtained the
decay rate $(1+t)^{-\frac{d}{4}+\frac{1}{2}}$ with $d\geq 2$ and $u_0\in L^1$.
Moreover, she conjectured that the sharp decay rate should be $(1+t)^{-\frac{d}{4}}$,~$d\geq 2$.
However, she failed to get it because she could not use the bootstrap argument as in \cite{Schonbek1985} due to the
additional stress tensor.  More recently, W. Luo and Z. Yin \cite{Luo-Yin,Luo-Yin2} improved the decay rate to $(1+t)^{-\frac{d}{4}}$ with $d\geq 2$.

In \cite{Luo-Yin}, W. Luo and Z. Yin also consider the general case($\sigma(u)=\nabla u$) and obtained the $L^2$
decay rate $(1+t)^{-\frac{d}{4}+\frac{1}{2}}$  for $d\geq 3$ and the
decay rate $\ln^{-l}(1+t)$ for $d=2$ with $u_0\in L^1$, $\sup_{R}\|\tilde{\psi}_0\|_{L^1}$. Obviously, this is not the optimal decay rate.
In this paper, we improve the decay rate to $(1+t)^{-\frac{d}{4}}$ for $d\geq 2$ with $u_0\in L^1$, $\tilde{\psi}_0\in\mathcal{L}^{q}(L^1)$.
Firstly, we use the Fourier splitting method after standard energy estimation. The most difficult for us is that the additional stress tensor $\tau$ does not decay fast enough. Thus, we failed to use the bootstrap argument as in \cite{Schonbek1985,Luo-Yin,Luo-Yin2}. To deal with this term, we need to use the coupling effect between $u$ and $\tilde{\psi}$. Motivated by \cite{He2009}, we take Fourier transform with respect to $x$ in \eqref{eq1}:
\begin{align}\label{1eq2}
\left\{
\begin{array}{ll}
\hat{u}_t+F(u\cdot\nabla u)+|\xi|^2 \hat{u}+i\xi\hat{P}=i\xi\cdot\hat{\tau},  \\[1ex]
(\hat{\tilde{\psi}})_t+F(u\cdot\nabla\tilde{\psi})+\mathcal{L}(\hat{\tilde{\psi}})=div_{R}(-F(\nabla u\cdot{R}\tilde{\psi}))+\psi_{\infty}(i\xi\otimes\hat{u})\cdot{R}\nabla_{R}\mathcal{U}, \\[1ex]
i\xi\cdot\bar{\hat{u}}=-\overline{i\xi\cdot\hat{u}}=0.
\end{array}
\right.
\end{align}
Multiplying $\bar{\hat{u}}(t,\xi)$ to the first equation of \eqref{1eq2} and taking the real part, since $\hat{P}i\xi\cdot\bar{\hat{u}}=0$, we obtain
\begin{align}
\frac 1 2 \frac d {dt} |\hat{u}(t,\xi)|^2+\mathcal{R}e[F(u\cdot\nabla u)\cdot\bar{\hat{u}}(t,\xi)]+|\xi|^2 |\hat{u}(t,\xi)|^2=\mathcal{R}e[i\xi\otimes\bar{\hat{u}}(t,\xi):\hat{\tau}].
\end{align}
Multiplying $\frac 1 {\psi_\infty} \overline{F(\tilde{\psi})}(t,\xi,R)$ to the second equation of \eqref{1eq2}, integrating over $B$ with $R$ and taking the real part, we get
\begin{align}
&\frac 1 2 \frac d {dt} \|\hat{\tilde{\psi}}\|^2_{\mathcal{L}^2}
+\mathcal{R}e[\int_{B}F(u\cdot\nabla\tilde{\psi})\cdot\frac 1 {\psi_\infty} \overline{F(\tilde{\psi})}(t,\xi,R)dR]+ \int_{B}\psi_\infty|\nabla_R \frac {\hat{\tilde{\psi}}} {\psi_\infty}|^2 dR  \\ \notag
&=\mathcal{R}e[i\xi\otimes\hat{u}:\bar{\hat{\tau}}]
+\mathcal{R}e[\int_{B}div_{R}(-F(\nabla u\cdot{R}\tilde{\psi}))\cdot\frac 1 {\psi_\infty} \overline{F(\tilde{\psi})}(t,\xi,R)dR].
\end{align}
The key observation is that $\mathcal{R}e[i\xi\otimes\bar{\hat{u}}(t,\xi):\hat{\tau}]+\mathcal{R}e[i\xi\otimes\hat{u}:\bar{\hat{\tau}}]=0$. 
Thus, we can cancel the stress term $\tau$ in Fourier space. Finally we obtain the $L^2$ decay rate $(1+t)^{-\frac{d}{4}}$ with $d\geq 3$ by virtue of the bootstrap argument. For $d=2$, we need an additional absorption method to prove the $L^2$ decay rate of solutions to the FENE model.

The paper is organized as follows. In Section 2, we introduce some notations and state some lemmas which will be useful in this paper. Finally, we give our main results. In Section 3, we study the $L^2$ decay rate of solutions to the FENE model for $d\geq 3$ by using the Fourier splitting method and the bootstrap argument. For $d=2$, we have to adjusted our method.

\section{Preliminaries and main results}
In this section we will introduce some notations and useful lemmas which will be used in the sequel.

If the function spaces are over $\mathbb{R}^d$ and $B$ with respect to the variable $x$ and $R$, for simplicity, we drop $\mathbb{R}^d$ and $B$ in the notation of function spaces if there is no ambiguity.

For $p\geq1$, we denote by $\mathcal{L}^{p}$ the space
$$\mathcal{L}^{p}=\big\{\psi \big|\|\psi\|^{p}_{\mathcal{L}^{p}}=\int_{B} \psi_{\infty}|\frac{\psi}{\psi_{\infty}}|^{p}dR<\infty\big\}.$$

We will use the notation $L^{p}_{x}(\mathcal{L}^{q})$ to denote $L^{p}[\mathbb{R}^{d};\mathcal{L}^{q}]:$
$$L^{p}_{x}(\mathcal{L}^{q})=\big\{\psi \big|\|\psi\|_{L^{p}_{x}(\mathcal{L}^{q})}=(\int_{\mathbb{R}^{d}}(\int_{B} \psi_{\infty}|\frac{\psi}{\psi_{\infty}}|^{q}dR)^{\frac{p}{q}}dx)^{\frac{1}{p}}<\infty\big\}.$$
When $p=q$, we also use the short notation $\mathcal{L}^p$ for $L^p_x(\mathcal{L}^{p})$ if there is no ambiguity.

The symbol $\widehat{f}=\mathcal{F}(f)$ denotes the Fourier transform of $f$.

Moreover, we denote by $\dot{\mathcal{H}}^1$ the space
$$\dot{\mathcal{H}}^1=\big\{\psi\big| \|\psi\|_{\dot{\mathcal{H}}^1}=(\int_B|\nabla_R \frac \psi {\psi_\infty}|^2\psi_\infty dR)^{\frac{1}{2}}\big\}.$$
We agree that $\nabla$ stands for $\nabla_x$ and $div$ stands for $div_x$.

\begin{lemm}\label{Lemma0}
For $d=2,~p\in[2,+\infty)$, then there exists a constant $C$ such that
 $$\|f\|_{L^{p}}\leq C \|f\|^{\frac 2 p}_{L^{2}}\|\nabla f\|^{\frac {p-2} p}_{L^{2}}.$$
For $d\geq 3$, then there exists a constant $C$ such that
 $$\|f\|_{L^{p}}\leq C \|\nabla f\|_{L^{2}}$$
where $p=\frac {2d} {d-2}$.
\end{lemm}

The following lemma allows us to estimate the extra stress tensor $\tau$.
\begin{lemm}\cite{Masmoudi2008}\label{Lemma1}
 If $\int_B \psi dR=0$ and
   $\displaystyle\int_B\bigg|\nabla _R  \bigg(\displaystyle\frac{\psi}{\psi _\infty }\bigg)\bigg|^2{\psi _\infty} dR<\infty$, then there exists a constant $C$ such that
   \[\int_{B}\frac{|\psi|^{2}}{\psi_{\infty}}dR\leq C \displaystyle\int_B\bigg|\nabla _R  \bigg(\displaystyle\frac{\psi}{\psi _\infty }\bigg)\bigg|^2{\psi _\infty} dR.\]
\end{lemm}
\begin{lemm}\label{Lemma2}
\cite{Masmoudi2008} For all $\varepsilon>0$, there exists a constant $C_{\varepsilon}$ such that
$$|\tau|^2\leq\varepsilon\int_{B}\psi_{\infty}|\nabla_{R}\frac{\psi}{\psi_{\infty}}|^{2}dR
+C_{\varepsilon}\int_{B}\frac{|\psi|^{2}}{\psi_{\infty}}dR,$$
\end{lemm}

The existence of the global solutions with small data was established in \cite{Masmoudi2008}.
\begin{theo}\label{th1}
Let $d\geq2$ and $s>1+\frac d 2$. Assume $u_0\in H^s$ and $\tilde{\psi}_0$ satisfies $\tilde{\psi}_0\in H^s(\mathcal{L}^2)$ and $\int_B\tilde{\psi}_0 dR=0$ ~$a.e.$ in $x$. A constant $\ep_0$ exists such that if
$$\|u_0\|^2_{H^s}+\|\tilde{\psi}_0\|^2_{H^s(\mathcal{L}^2)}\leq \ep_0, $$
then there exists a unique global solution $(u,\tilde{\psi})$ of \eqref{eq1} such that $u\in C(R^{+}; H^s)\cap L^2_{loc}(R^{+}; H^{s+1})$ and $\tilde{\psi}\in C(R^{+}; H^s(\mathcal{L}^2))\cap L^2_{loc}(R^{+}; H^{s}(\dot{\mathcal{H}}^1))$ with $\int_B \tilde{\psi}dR=0$~$a.e.$ in $x$. Moreover
$$\|u\|^2_{H^s}+\|\tilde{\psi}\|^2_{H^s(\mathcal{L}^2)}+\int_{0}^{\infty}\|\nabla u\|^2_{H^s}+\|\tilde{\psi}\|^2_{H^s(\dot{\mathcal{H}}^1)}dt\leq C\ep_0, $$
where $C$ is a constant dependent on the initial data.
\end{theo}

Our main result can be stated as follows.
\begin{theo}\label{th2}
Let $d\geq2$ and $s>1+\frac d 2$. Let $(u,\tilde{\psi})$ be the strong solution of \eqref{eq1} with the initial data $(u_0,\tilde{\psi}_0)$ under the condition of Theorem \eqref{th1}. In addition, if $u_0\in L^1$ and $\tilde{\psi}_0\in \mathcal{L}^2(L^1)$. There exists a constant $C$ such that
\begin{align}
\|u\|_{L^2}\leq C(1+t)^{-\frac{d}{4}},
\end{align}
\begin{align}
\|\tilde{\psi}\|_{L^2(\mathcal{L}^2)}\leq C(1+t)^{-\frac{d}{4}-\frac{1}{2}}.
\end{align}
\end{theo}
\begin{rema}
In \cite{Schonbek1991}, M. Schonbek showed that $(1+t)^{-\frac{d}{4}},\, d\geq 2$ is the optimal $L^2$ decay rate for the Navier-Stokes equations with $u_0\in L^1$. Note that
if $\tilde{\psi}$ is independent on $x$, then $div~\tau=0$. Then, the FENE model is reduced to the Navier-Stokes equations. Thus, the $L^2$ decay rate of $u$ for the FENE model which we obtained in Theorem \ref{th2} is sharp for all $d\geq 2$.
\end{rema}

\section{The $L^2$ decay rate}
This section is devoted to investigating the long time behaviour for the velocity of the FENE dumbbell model. The most difficult for us is that the additional stress tensor $\tau$ does not decay fast enough. Thus, we failed to use the bootstrap argument as in \cite{Schonbek1985,Luo-Yin,Luo-Yin2}. To deal with this term, we need to use the coupling effect between $u$ and $\tilde{\psi}$. Motivated by \cite{He2009} for $d=3$, we obtain the $L^2$ decay rate in Theorem \eqref{th2} by taking Fourier transform in \eqref{eq1} and using the bootstrap argument.

Let us first state our main result concerning the $L^2$ decay rate of the global classical solutions for $d\geq3$.
\begin{prop}\label{pro1}
Let $(u,\tilde{\psi})$ be the strong solution of \eqref{eq1} with the initial data $(u_0,\tilde{\psi}_0)$ under the condition of Theorem \eqref{th2} for $d\geq3$. There exists a constant $C$ such that
\begin{align}
\|u\|_{L^2}\leq C(1+t)^{-\frac{d}{4}},
\end{align}
\begin{align}
\|\tilde{\psi}\|_{L^2(\mathcal{L}^2)}\leq C(1+t)^{-\frac{d}{4}-\frac{1}{2}}.
\end{align}
\end{prop}
\begin{proof}
Using standard energy estimation to \eqref{eq1}, we have
\begin{align*}
&\frac d {dt} (\|u\|^2_{L^2}+\|\tilde{\psi}\|^2_{L^2(\mathcal{L}^2)})+2\|\nabla u\|^2_{L^2}+2\int_{\mathbb{R}^{d}}\int_{B}\psi_{\infty}|\nabla_{R}\frac{\tilde{\psi}}{\psi_{\infty}}|^{2}dRdx  \\
&=2\int_{\mathbb{R}^{d}}\int_{B}\nabla_R\cdot(\nabla u R\tilde{\psi})\frac{\tilde{\psi}}{\psi_{\infty}}dRdx.
\end{align*}
Integrating by part and using Lemma \eqref{Lemma1}, Theorem \eqref{th1}, we get
\begin{align*}
\int_{\mathbb{R}^{d}}\int_{B}\nabla_R\cdot(\nabla u R\tilde{\psi})\frac{\tilde{\psi}}{\psi_{\infty}}dRdx
=\int_{\mathbb{R}^{d}}\int_{B}\psi_{\infty}\nabla u R\frac{\tilde{\psi}}{\psi_{\infty}}
\nabla_R\frac{\tilde{\psi}}{\psi_{\infty}}dRdx\leq C\ep_0 \int_{\mathbb{R}^{d}}\int_{B}\psi_{\infty}|\nabla_{R}\frac{\tilde{\psi}}{\psi_{\infty}}|^{2}dRdx,
\end{align*}
which implies that
\begin{align}\label{ineq0}
\frac d {dt} (\|u\|^2_{L^2}+\|\tilde{\psi}\|^2_{L^2(\mathcal{L}^2)})+\|\nabla u\|^2_{L^2}
+\int_{\mathbb{R}^{d}}\int_{B}\psi_{\infty}|\nabla_{R}\frac{\tilde{\psi}}{\psi_{\infty}}|^{2}dRdx\leq 0.
\end{align}
Define $S(t)=\{\xi:|\xi|^2\leq C_d(1+t)^{-1}\}$ where the constant $C_d$ will be chosen later on. Using Schonbek's strategy, we split the phase space into two time-dependent domain:
$$\|\nabla u\|^2_{L^2}=\int_{S(t)}|\xi|^2|\hat{u}(\xi)|^2 d\xi+\int_{S(t)^c}|\xi|^2|\hat{u}(\xi)|^2 d\xi.$$
Then we can easily deduce that
$$\frac {C_d} {1+t} \int_{S(t)^c}|\hat{u}(\xi)|^2 d\xi\leq\|\nabla u\|^2_{L^2},$$
which implies that
\begin{align}\label{ineq1}
\frac d {dt} (\|u\|^2_{L^2}+\|\tilde{\psi}\|^2_{L^2(\mathcal{L}^2)})+\frac {C_d} {1+t}\| u\|^2_{L^2}
+\int_{\mathbb{R}^{d}}\int_{B}\psi_{\infty}|\nabla_{R}\frac{\tilde{\psi}}{\psi_{\infty}}|^{2}dRdx\leq \frac {C_d} {1+t}\int_{S(t)}|\hat{u}(\xi)|^2 d\xi.
\end{align}
From now on, we focus on the $L^2$ estimate to the low frequency part of $u$. We take Fourier transform with respect to $x$ in \eqref{eq1}:
\begin{align}\label{eq2}
\left\{
\begin{array}{ll}
\hat{u}_t+F(u\cdot\nabla u)+|\xi|^2 \hat{u}+i\xi\hat{P}=i\xi\cdot\hat{\tau},  \\[1ex]
(\hat{\tilde{\psi}})_t+F(u\cdot\nabla\tilde{\psi})+\mathcal{L}(\hat{\tilde{\psi}})=div_{R}(-F(\nabla u\cdot{R}\tilde{\psi}))+\psi_{\infty}(i\xi\otimes\hat{u})\cdot{R}\nabla_{R}\mathcal{U}, \\[1ex]
i\xi\cdot\bar{\hat{u}}=-\overline{i\xi\cdot\hat{u}}=0.
\end{array}
\right.
\end{align}
Multiplying $\bar{\hat{u}}(t,\xi)$ to the first equation of \eqref{eq2} and taking the real part, since $\hat{P}i\xi\cdot\bar{\hat{u}}=0$, we obtain
\begin{align}\label{eq3}
\frac 1 2 \frac d {dt} |\hat{u}(t,\xi)|^2+\mathcal{R}e[F(u\cdot\nabla u)\cdot\bar{\hat{u}}(t,\xi)]+|\xi|^2 |\hat{u}(t,\xi)|^2=\mathcal{R}e[i\xi\otimes\bar{\hat{u}}(t,\xi):\hat{\tau}].
\end{align}
Multiplying $\frac 1 {\psi_\infty} \overline{F(\tilde{\psi})}(t,\xi,R)$ to the second equation of \eqref{eq2}, integrating over $B$ with $R$ and taking the real part, we get
\begin{align}\label{eq4}
&\frac 1 2 \frac d {dt} \|\hat{\tilde{\psi}}\|^2_{\mathcal{L}^2}
+\mathcal{R}e[\int_{B}F(u\cdot\nabla\tilde{\psi})\cdot\frac 1 {\psi_\infty} \overline{F(\tilde{\psi})}(t,\xi,R)dR]+ \int_{B}\psi_\infty|\nabla_R \frac {\hat{\tilde{\psi}}} {\psi_\infty}|^2 dR  \\ \notag
&=\mathcal{R}e[i\xi\otimes\hat{u}:\bar{\hat{\tau}}]
+\mathcal{R}e[\int_{B}div_{R}(-F(\nabla u\cdot{R}\tilde{\psi}))\cdot\frac 1 {\psi_\infty} \overline{F(\tilde{\psi})}(t,\xi,R)dR].
\end{align}
We can easily deduce that
$$-\mathcal{R}e[F(u\cdot\nabla u)\cdot\bar{\hat{u}}(t,\xi)]=-\mathcal{R}e[F(u\otimes u):i\xi\otimes\bar{\hat{u}}(t,\xi)]\leq \frac 1 2 (|F(u\otimes u)|^2+|\xi|^2 |\hat{u}(t,\xi)|^2),$$
and
$$\mathcal{R}e[i\xi\otimes\bar{\hat{u}}(t,\xi):\hat{\tau}]+\mathcal{R}e[i\xi\otimes\hat{u}:\bar{\hat{\tau}}]=0.$$
Using Lemma \eqref{eq1}, we get
\begin{align*}
\mathcal{R}e[\int_{B}F(u\cdot\nabla\tilde{\psi})\cdot\frac 1 {\psi_\infty} \overline{F(\tilde{\psi})}(t,\xi,R)dR]\leq
C_\delta\int_{B}\psi_\infty|F(u\cdot\nabla\frac {\tilde{\psi}} {\psi_\infty})|^2 dR+\delta\int_{B}\psi_\infty|\nabla_R \frac {\hat{\tilde{\psi}}} {\psi_\infty}|^2 dR,
\end{align*}
where $\delta$ is small enough.
Integrating by part and using Lemma \eqref{eq1}, we get
\begin{align*}
&\mathcal{R}e[\int_{B}div_{R}(-F(\nabla u\cdot{R}\tilde{\psi}))\cdot\frac 1 {\psi_\infty} \overline{F(\tilde{\psi})}(t,\xi,R)dR]=\mathcal{R}e[\int_{B}F(\nabla u\cdot{R}\tilde{\psi})\nabla_R (\frac 1 {\psi_\infty} \overline{F(\tilde{\psi})}(t,\xi,R))dR] \\ \notag
&\leq C_\delta\int_{B}\psi_\infty|F(\nabla u\cdot{R}\frac {\tilde{\psi}} {\psi_\infty})|^2 dR+\delta\int_{B}\psi_\infty|\nabla_R \frac {\hat{\tilde{\psi}}} {\psi_\infty}|^2 dR.
\end{align*}
Plugging the above estimates into \eqref{eq3} and \eqref{eq4}, then we have
\begin{align}\label{ineq2}
&\frac d {dt} (|\hat{u}(t,\xi)|^2+\|\hat{\tilde{\psi}}\|^2_{\mathcal{L}^2})
+|\xi|^2 |\hat{u}(t,\xi)|^2+\int_{B}\psi_\infty|\nabla_R \frac {\hat{\tilde{\psi}}} {\psi_\infty}|^2 dR  \\ \notag
&\leq C|F(u\otimes u)|^2+C_\delta\int_{B}\psi_\infty|F(u\cdot\nabla\frac {\tilde{\psi}} {\psi_\infty})|^2 dR+C_\delta\int_{B}\psi_\infty|F(\nabla u\cdot{R}\frac {\tilde{\psi}} {\psi_\infty})|^2 dR.
\end{align}
Consider $\xi\in S(t)$ and sufficiently large $t$ satisfies $CC_d\leq(1+t)$, using Lemma \eqref{Lemma1}, then we get
\begin{align}\label{ineq3}
|\xi|^2\|\hat{\tilde{\psi}}\|^2_{\mathcal{L}^2}\leq\int_{B}\psi_\infty|\nabla_R \frac {\hat{\tilde{\psi}}} {\psi_\infty}|^2 dR.
\end{align}
From \eqref{ineq2} and \eqref{ineq3}, we deduce that
\begin{align}\label{ineq4}
&|\hat{u}(t,\xi)|^2+\|\hat{\tilde{\psi}}\|^2_{\mathcal{L}^2}
\leq e^{-|\xi|^2 t}(|\hat{u}_0|^2
+\|\hat{\tilde{\psi}}_0\|^2_{\mathcal{L}^2})+C\int_{0}^{t}|F(u\otimes u)|^2 ds  \\ \notag
&+C_\delta\int_{0}^{t}\int_{B}\psi_\infty|F(u\cdot\nabla\frac {\tilde{\psi}} {\psi_\infty})|^2 dRds+C_\delta\int_{0}^{t}\int_{B}\psi_\infty|F(\nabla u\cdot{R}\frac {\tilde{\psi}} {\psi_\infty})|^2 dRds.
\end{align}
Integrating over $S(t)$ with $\xi$, then we have
\begin{align}\label{ineq5}
&\int_{S(t)}|\hat{u}(t,\xi)|^2+\|\hat{\tilde{\psi}}\|^2_{\mathcal{L}^2}d\xi
\leq \int_{S(t)} e^{-|\xi|^2 t}(|\hat{u}_0|^2
+\|\hat{\tilde{\psi}}_0\|^2_{\mathcal{L}^2})d\xi+C\int_{S(t)}\int_{0}^{t}|F(u\otimes u)|^2 dsd\xi  \\ \notag
&+C_\delta\int_{S(t)}\int_{0}^{t}\int_{B}\psi_\infty|F(u\cdot\nabla\frac {\tilde{\psi}} {\psi_\infty})|^2 dRdsd\xi+C_\delta\int_{S(t)}\int_{0}^{t}\int_{B}\psi_\infty|F(\nabla u\cdot{R}\frac {\tilde{\psi}} {\psi_\infty})|^2 dRdsd\xi.
\end{align}
Then with the additional assumption in Theorem \eqref{th2}, we obtain
\begin{align*}
\int_{S(t)} e^{-|\xi|^2 t}(|\hat{u}_0|^2
+\|\hat{\tilde{\psi}}_0\|^2_{\mathcal{L}^2})d\xi
&\leq\int_{S(t)} e^{-|\xi|^2 t}d\xi\cdot\||\hat{u}_0|^2
+\|\hat{\tilde{\psi}}_0\|^2_{\mathcal{L}^2}\|_{L^{\infty}(S(t))} \\
&\leq C(1+t)^{-\frac d 2}(\|u_0\|^2_{L^1}+\|\tilde{\psi}_0\|^2_{\mathcal{L}^2(L^1)}).
\end{align*}
Thanks to Minkowski's inequality and Theorem \eqref{th1}, we get
\begin{align}\label{ineq6}
\int_{S(t)}\int_{0}^{t}|F(u\otimes u)|^2 dsd\xi
&=\int_{0}^{t}\int_{S(t)}|F(u\otimes u)|^2 d\xi ds  \\ \notag
&\leq C\int_{S(t)}d\xi \int_{0}^{t}\||F(u\otimes u)|^2\|_{L^{\infty}}ds \\ \notag
&\leq C(1+t)^{-\frac d 2} \int_{0}^{t}\|u\|^4_{L^{2}}ds \\ \notag
&\leq C(1+t)^{-\frac d 2+1}.
\end{align}
Using Theorem \eqref{th1} and Lemma \eqref{Lemma1}, we get
\begin{align}\label{ineq7}
\int_{S(t)}\int_{0}^{t}\int_{B}\psi_\infty|F(u\cdot\nabla\frac {\tilde{\psi}} {\psi_\infty})|^2 dRdsd\xi
&\leq C\int_{S(t)}d\xi \int_{0}^{t}\int_{B}\|\psi_\infty|F(u\cdot\nabla\frac {\tilde{\psi}} {\psi_\infty})|^2\|_{L^{\infty}}dRds \\ \notag
&\leq C(1+t)^{-\frac d 2} \int_{0}^{t}\|u\|^2_{L^{2}}\|\nabla\tilde{\psi}\|^2_{L^{2}(\mathcal{L}^{2})}ds \\ \notag
&\leq C(1+t)^{-\frac d 2}.
\end{align}
A similar argument for the last term in \eqref{ineq5} yields
\begin{align}\label{ineq8}
\int_{S(t)}\int_{0}^{t}\int_{B}\psi_\infty|F(\nabla u\cdot{R}\frac {\tilde{\psi}} {\psi_\infty})|^2 dRdsd\xi
&\leq C\int_{S(t)}d\xi \int_{0}^{t}\int_{B}\|\psi_\infty|F(\nabla u\cdot{R}\frac {\tilde{\psi}} {\psi_\infty})|^2\|_{L^{\infty}}dRds \\ \notag
&\leq C(1+t)^{-\frac d 2} \int_{0}^{t}\|\nabla u\|^2_{L^{2}}\|\tilde{\psi}\|^2_{L^{2}(\mathcal{L}^{2})}ds \\ \notag
&\leq C(1+t)^{-\frac d 2}.
\end{align}
Plugging the above estimates into \eqref{ineq5}, we obtain
\begin{align}\label{ineq9}
\int_{S(t)}|\hat{u}(t,\xi)|^2 d\xi\leq C(1+t)^{-\frac d 2+1}.
\end{align}
According to \eqref{ineq1} and \eqref{ineq9}, we deduce that
\begin{align*}
\frac d {dt} (\|u\|^2_{L^2}+\|\tilde{\psi}\|^2_{L^2(\mathcal{L}^2)})+\frac {C_d} {1+t}\| u\|^2_{L^2}
+\int_{\mathbb{R}^{d}}\int_{B}\psi_{\infty}|\nabla_{R}\frac{\tilde{\psi}}{\psi_{\infty}}|^{2}dRdx\leq \frac {C_d} {1+t} C(1+t)^{-\frac d 2+1},
\end{align*}
from which we deduce that if $C_d\geq \frac d 2 +2$, then
\begin{align}\label{ineq10}
\|u\|^2_{L^2}+\|\tilde{\psi}\|^2_{L^2(\mathcal{L}^2)}\leq  C(1+t)^{-\frac d 2+1},
\end{align}
where $\frac d 2-1>0$.
With the above $L^2$ decay estimate for $u$, we can improve the $L^2$ decay rate in \eqref{ineq10} by using the bootstrap argument.  \\
If $d\geq4$, then we have
\begin{align*}
\int_{S(t)}\int_{0}^{t}|F(u\otimes u)|^2 dsd\xi
&\leq C(1+t)^{-\frac d 2} \int_{0}^{t}\|u\|^4_{L^{2}}ds \\ \notag
&\leq C(1+t)^{-\frac d 2} \int_{0}^{t}(1+s)^{-d+2}ds \\ \notag
&\leq C(1+t)^{-\frac d 2}.
\end{align*}
Then the proof of \eqref{ineq10} implies that
\begin{align}\label{ineq11}
\|u\|_{L^2}+\|\tilde{\psi}\|_{L^2(\mathcal{L}^2)}\leq  C(1+t)^{-\frac d 4},~~for~d\geq4.
\end{align}
If $d=3$, using \eqref{ineq10}, then we have
\begin{align*}
\int_{S(t)}\int_{0}^{t}|F(u\otimes u)|^2 dsd\xi
&\leq C(1+t)^{-\frac 3 2} \int_{0}^{t}\|u\|^4_{L^{2}}ds \\ \notag
&\leq C(1+t)^{-\frac 3 2} \int_{0}^{t}(1+s)^{-1}ds \\ \notag
&\leq C(1+t)^{-1}.
\end{align*}
Then the proof of \eqref{ineq10} implies that
\begin{align}\label{ineq12}
\|u\|^2_{L^2}+\|\tilde{\psi}\|^2_{L^2(\mathcal{L}^2)}\leq  C(1+t)^{-1},~~for~d=3.
\end{align}
Using \eqref{ineq12}, then we have
\begin{align*}
\int_{S(t)}\int_{0}^{t}|F(u\otimes u)|^2 dsd\xi
&\leq C(1+t)^{-\frac 3 2} \int_{0}^{t}\|u\|^4_{L^{2}}ds \\ \notag
&\leq C(1+t)^{-\frac 3 2} \int_{0}^{t}(1+s)^{-2}ds \\ \notag
&\leq C(1+t)^{-\frac 3 2}.
\end{align*}
Then the proof of \eqref{ineq10} implies that
\begin{align}\label{ineq13}
\|u\|_{L^2}+\|\tilde{\psi}\|_{L^2(\mathcal{L}^2)}\leq  C(1+t)^{-\frac 3 4},~~for~d=3.
\end{align}
We get the $L^2$ decay rate for $u$ from \eqref{ineq11} and \eqref{ineq13}.
To improve the decay rate for $\|\tilde{\psi}\|_{L^2(\mathcal{L}^2)}$, we need to consider the $L^2$ decay rate for $\nabla u$. \\
We claim that
\begin{align}\label{ineq14}
\|\nabla u\|_{L^2}+\|\nabla\tilde{\psi}\|_{L^2(\mathcal{L}^2)}+\|\tilde{\psi}\|_{L^2(\mathcal{L}^2)}\leq  C(1+t)^{-\frac d 4-\frac 1 2}.
\end{align}
Applying the operator $\nabla^{\alpha}$ with $|\alpha|=1$ to \eqref{eq1} and using standard energy estimation, we have
\begin{align*}
&\frac 1 2 \frac d {dt} (\|\nabla^{\alpha}u\|^2_{L^2}+\|\nabla^{\alpha}\tilde{\psi}\|^2_{L^2(\mathcal{L}^2)})+\|\nabla \nabla^{\alpha} u\|^2_{L^2}+\int_{\mathbb{R}^{d}}\int_{B}\psi_{\infty}|\nabla_{R}\frac{\nabla^{\alpha}\tilde{\psi}}{\psi_{\infty}}|^{2}dRdx  \\
&=\int_{\mathbb{R}^{d}}\int_{B}\nabla^{\alpha}\nabla_R\cdot(\nabla u R\tilde{\psi})\frac{\nabla^{\alpha}\tilde{\psi}}{\psi_{\infty}}dRdx
+\int_{\mathbb{R}^{d}}(\nabla^{\alpha}u\cdot\nabla)u\cdot\nabla^{\alpha}udx
+\int_{\mathbb{R}^{d}}\int_{B}(\nabla^{\alpha}u\cdot\nabla)\tilde{\psi}\nabla^{\alpha}\tilde{\psi}dRdx.
\end{align*}
Integrating by part and using Lemma \eqref{Lemma1}, Theorem \eqref{th1}, we get
\begin{align*}
\int_{\mathbb{R}^{d}}\int_{B}\nabla^{\alpha}\nabla_R\cdot(\nabla u R\tilde{\psi})\frac{\nabla^{\alpha}\tilde{\psi}}{\psi_{\infty}}dRdx\leq C\ep_0 (\|\nabla \nabla^{\alpha} u\|^2_{L^2}+\int_{\mathbb{R}^{d}}\int_{B}\psi_{\infty}|\nabla_{R}\frac{\nabla^{\alpha}\tilde{\psi}}{\psi_{\infty}}|^{2}dRdx).
\end{align*}
Integrating by part and using Theorem \eqref{th1}, we have
\begin{align*}
&\int_{\mathbb{R}^{d}}(\nabla^{\alpha}u\cdot\nabla)u\cdot\nabla^{\alpha}udx=-\int_{\mathbb{R}^{d}}(\nabla^{\alpha}u\cdot\nabla)\nabla^{\alpha}u\cdot udx \\
&\leq \|\nabla \nabla^{\alpha} u\|_{L^2}\|\nabla^{\alpha} u\|_{L^{\frac {2d} {d-2}}}\|u\|_{L^d} \leq C\ep_0 \|\nabla \nabla^{\alpha} u\|^2_{L^2}.
\end{align*}
Using Lemma \eqref{Lemma0}, Lemma \eqref{Lemma1} and Theorem \eqref{th1}, we get
\begin{align*}
\int_{\mathbb{R}^{d}}\int_{B}(\nabla^{\alpha}u\cdot\nabla)\tilde{\psi}\nabla^{\alpha}\tilde{\psi}dRdx
&\leq \|\nabla\tilde{\psi}\|_{L^d(\mathcal{L}^2)}\|\nabla^{\alpha} u\|_{L^{\frac {2d} {d-2}}}
\|\nabla^{\alpha}\tilde{\psi}\|^2_{L^2(\mathcal{L}^2)}  \\
&\leq C\ep_0 (\|\nabla \nabla^{\alpha} u\|^2_{L^2}+\int_{\mathbb{R}^{d}}\int_{B}\psi_{\infty}|\nabla_{R}\frac{\nabla^{\alpha}\tilde{\psi}}{\psi_{\infty}}|^{2}dRdx).
\end{align*}
Summing up all the estimates for $|\alpha|=1$, we deduce that
\begin{align}\label{ineq15}
\frac d {dt} (\|\nabla u\|^2_{L^2}+\|\nabla\tilde{\psi}\|^2_{L^2(\mathcal{L}^2)})+\|\nabla^2 u\|^2_{L^2}
+\int_{\mathbb{R}^{d}}\int_{B}\psi_{\infty}|\nabla_{R}\frac{\nabla\tilde{\psi}}{\psi_{\infty}}|^{2}dRdx\leq 0,
\end{align}
which implies that
\begin{align}\label{ineq16}
\frac d {dt} (\|\nabla u\|^2_{L^2}+\|\nabla\tilde{\psi}\|^2_{L^2(\mathcal{L}^2)})+\frac {C_d} {1+t}\|\nabla u\|^2_{L^2}
+\int_{\mathbb{R}^{d}}\int_{B}\psi_{\infty}|\nabla_{R}\frac{\nabla\tilde{\psi}}{\psi_{\infty}}|^{2}dRdx\leq \frac {C_d} {1+t}\int_{S(t)}|\xi|^2 |\hat{u}(\xi)|^2 d\xi.
\end{align}
According to \eqref{ineq11} and \eqref{ineq13}, we have
\begin{align*}
\frac {C_d} {1+t}\int_{S(t)}|\xi|^2 |\hat{u}(\xi)|^2 d\xi\leq {C_d}^2 (1+t)^{-2}\|u\|^2_{L^2}\leq C (1+t)^{-\frac d 2-2}.
\end{align*}
This together with \eqref{ineq16} ensures that
\begin{align*}
\|\nabla u\|_{L^2}+\|\nabla\tilde{\psi}\|_{L^2(\mathcal{L}^2)}\leq C (1+t)^{-\frac d 4-\frac 1 2}.
\end{align*}
Using standard energy estimation to the second equation in \eqref{eq1}, we get
\begin{align*}
&\frac d {dt} \|\tilde{\psi}\|^2_{L^2(\mathcal{L}^2)}+2\int_{\mathbb{R}^{d}}\int_{B}\psi_{\infty}|\nabla_{R}\frac{\tilde{\psi}}{\psi_{\infty}}|^{2}dRdx  \\
&=2\int_{\mathbb{R}^{d}}\int_{B}\nabla_R\cdot(\nabla u R\tilde{\psi})\frac{\tilde{\psi}}{\psi_{\infty}}dRdx+2\int_{\mathbb{R}^{d}}\int_{B}\nabla u\cdot{R}\nabla_{R}\mathcal{U}\tilde{\psi}dRdx.
\end{align*}
Using Lemma \eqref{Lemma1}, Lemma \eqref{Lemma2} and Theorem \eqref{th1}, we have
\begin{align*}
\int_{\mathbb{R}^{d}}\int_{B}\nabla_R\cdot(\nabla u R\tilde{\psi})\frac{\tilde{\psi}}{\psi_{\infty}}dRdx
=\int_{\mathbb{R}^{d}}\int_{B}\psi_{\infty}\nabla u R\frac{\tilde{\psi}}{\psi_{\infty}}
\nabla_R\frac{\tilde{\psi}}{\psi_{\infty}}dRdx\leq C\ep_0 \int_{\mathbb{R}^{d}}\int_{B}\psi_{\infty}|\nabla_{R}\frac{\tilde{\psi}}{\psi_{\infty}}|^{2}dRdx
\end{align*}
and
\begin{align*}
\int_{\mathbb{R}^{d}}\int_{B}\nabla u\cdot{R}\nabla_{R}\mathcal{U}\tilde{\psi}dRdx
=\int_{\mathbb{R}^{d}}\nabla u\cdot\tau dx\leq \delta\int_{\mathbb{R}^{d}}\int_{B}\psi_{\infty}|\nabla_{R}\frac{\tilde{\psi}}{\psi_{\infty}}|^{2}dRdx+C_\delta \|\nabla u\|^2_{L^2}.
\end{align*}
Using Lemma \eqref{Lemma1}, we obtain
\begin{align*}
\frac d {dt} \|\tilde{\psi}\|^2_{L^2(\mathcal{L}^2)}+\|\tilde{\psi}\|^2_{L^2(\mathcal{L}^2)}\leq C_\delta \|\nabla u\|^2_{L^2},
\end{align*}
from which we deduce that
\begin{align*}
\|\tilde{\psi}\|^2_{L^2(\mathcal{L}^2)}
&\leq \|\tilde{\psi}_0\|^2_{L^2(\mathcal{L}^2)}e^{-t}+C\int_{0}^{t}e^{-(t-s)}\|\nabla u\|^2_{L^2}ds  \\
&\leq C(e^{-t}+\int_{0}^{t}e^{-(t-s)}(1+s)^{-\frac d 2-1}ds)  \\
&\leq C(1+t)^{-\frac d 2-1},
\end{align*}
where in the last inequality we have used the fact that
\begin{align*}
&\lim_{t\rightarrow\infty}(1+t)^{\frac d 2+1}\int_{0}^{t}e^{-(t-s)}(1+s)^{-\frac d 2-1}ds = \lim_{t\rightarrow\infty}\frac {(1+t)^{\frac d 2+1}\int_{0}^{t}e^{s}(1+s)^{-\frac d 2-1}ds} {e^{t}}  \\
&=1+\lim_{t\rightarrow\infty}\frac {(\frac d 2+1)(1+t)^{\frac d 2}\int_{0}^{t}e^{s}(1+s)^{-\frac d 2-1}ds} {e^{t}}   \\
&=1.
\end{align*}
Finally, we prove the claim \eqref{ineq14}.
We thus complete the proof of Proposition \eqref{pro1}.
\end{proof}

By \eqref{ineq14}, Proposition \eqref{pro1} and Lemma \eqref{Lemma0}, we obtain the following corollary.
\begin{coro}
Under the assumption of Theorem \eqref{th2}, for $2\leq p\leq \frac {2d} {d-2}$, we have
\begin{align}
\|u\|_{L^p}\leq C(1+t)^{-\frac{d}{2}(1-\frac 1 p)},~~
\|\tilde{\psi}\|_{L^p(\mathcal{L}^2)}\leq C(1+t)^{-\frac{d}{4}-\frac{1}{2}}.
\end{align}
\end{coro}

Finally, for $d=2$, we have the following main result concerning the $L^2$ decay rate of the global classical solutions .
\begin{prop}\label{pro2}
Let $(u,\tilde{\psi})$ be the strong solution of \eqref{eq1} with the initial data $(u_0,\tilde{\psi}_0)$ under the condition of Theorem \eqref{th2} for $d=2$. In addition, if $u_0\in L^1$ and $\tilde{\psi}_0\in \mathcal{L}^2(L^1)$. There exists a constant $C$ such that
\begin{align}
\|u\|_{L^2}\leq C(1+t)^{-\frac{1}{2}},
\end{align}
\begin{align}
\|\tilde{\psi}\|_{L^2(\mathcal{L}^2)}\leq C(1+t)^{-1}.
\end{align}
\end{prop}
\begin{proof}
Similar to the proof of Proposition \eqref{pro1}, we have the following standard energy estimation to \eqref{eq1}:
\begin{align}\label{ineq17}
\frac d {dt} (\|u\|^2_{L^2}+\|\tilde{\psi}\|^2_{L^2(\mathcal{L}^2)})+\|\nabla u\|^2_{L^2}
+\int_{\mathbb{R}^{2}}\int_{B}\psi_{\infty}|\nabla_{R}\frac{\tilde{\psi}}{\psi_{\infty}}|^{2}dRdx\leq 0.
\end{align}
Define $S(t)=\{\xi:|\xi|^2\leq 3(1+t)^{-1}\}$. According to Schonbek's strategy, we obtain
\begin{align}\label{ineq18}
\frac d {dt} (\|u\|^2_{L^2}+\|\tilde{\psi}\|^2_{L^2(\mathcal{L}^2)})+\frac {3} {1+t}\| u\|^2_{L^2}
+\int_{\mathbb{R}^{2}}\int_{B}\psi_{\infty}|\nabla_{R}\frac{\tilde{\psi}}{\psi_{\infty}}|^{2}dRdx\leq \frac {3} {1+t}\int_{S(t)}|\hat{u}(\xi)|^2 d\xi.
\end{align}
Now, we need to consider the $L^2$ estimate to the low frequency part of $u$. First, we have the following standard estimation to \eqref{eq2}:
\begin{align}\label{ineq19}
&\frac d {dt} (|\hat{u}(t,\xi)|^2+\|\hat{\tilde{\psi}}\|^2_{\mathcal{L}^2})
+|\xi|^2 |\hat{u}(t,\xi)|^2+\int_{B}\psi_\infty|\nabla_R \frac {\hat{\tilde{\psi}}} {\psi_\infty}|^2 dR  \\ \notag
&\leq C|F(u\otimes u)|^2+C_\delta\int_{B}\psi_\infty|F(u\cdot\nabla\frac {\tilde{\psi}} {\psi_\infty})|^2 dR+C_\delta\int_{B}\psi_\infty|F(\nabla u\cdot{R}\frac {\tilde{\psi}} {\psi_\infty})|^2 dR.
\end{align}
Consider $\xi\in S(t)$ and sufficiently large $t$ satisfies $3C\leq(1+t)$, using Lemma \eqref{Lemma1}, then we get
\begin{align}\label{ineq20}
|\xi|^2\|\hat{\tilde{\psi}}\|^2_{\mathcal{L}^2}\leq\int_{B}\psi_\infty|\nabla_R \frac {\hat{\tilde{\psi}}} {\psi_\infty}|^2 dR.
\end{align}
Using \eqref{ineq19}, \eqref{ineq20} and integrating over $S(t)$ with $\xi$, then we have
\begin{align}\label{ineq21}
&\int_{S(t)}|\hat{u}(t,\xi)|^2+\|\hat{\tilde{\psi}}\|^2_{\mathcal{L}^2}d\xi
\leq \int_{S(t)} e^{-|\xi|^2 t}(|\hat{u}_0|^2
+\|\hat{\tilde{\psi}}_0\|^2_{\mathcal{L}^2})d\xi+C\int_{S(t)}\int_{0}^{t}|F(u\otimes u)|^2 dsd\xi  \\ \notag
&+C_\delta\int_{S(t)}\int_{0}^{t}\int_{B}\psi_\infty|F(u\cdot\nabla\frac {\tilde{\psi}} {\psi_\infty})|^2 dRdsd\xi+C_\delta\int_{S(t)}\int_{0}^{t}\int_{B}\psi_\infty|F(\nabla u\cdot{R}\frac {\tilde{\psi}} {\psi_\infty})|^2 dRdsd\xi.
\end{align}
Then with the additional assumption in Theorem \eqref{th2}, we obtain
\begin{align*}
\int_{S(t)} e^{-|\xi|^2 t}(|\hat{u}_0|^2
+\|\hat{\tilde{\psi}}_0\|^2_{\mathcal{L}^2})d\xi
&\leq\int_{S(t)} e^{-|\xi|^2 t}d\xi\cdot\||\hat{u}_0|^2
+\|\hat{\tilde{\psi}}_0\|^2_{\mathcal{L}^2}\|_{L^{\infty}(S(t))} \\
&\leq C(1+t)^{-1}(\|u_0\|^2_{L^1}+\|\tilde{\psi}_0\|^2_{\mathcal{L}^2(L^1)}).
\end{align*}
Thanks to Minkowski's inequality and Theorem \eqref{th1}, we get
\begin{align}\label{ineq22}
\int_{S(t)}\int_{0}^{t}|F(u\otimes u)|^2 dsd\xi
&=\int_{0}^{t}\int_{S(t)}|F(u\otimes u)|^2 d\xi ds  \\ \notag
&\leq C\int_{S(t)}d\xi \int_{0}^{t}\||F(u\otimes u)|^2\|_{L^{\infty}}ds \\ \notag
&\leq C(1+t)^{-1} \int_{0}^{t}\|u\|^4_{L^{2}}ds.
\end{align}
Using Theorem \eqref{th1} and Lemma \eqref{Lemma1}, we get
\begin{align}\label{ineq23}
\int_{S(t)}\int_{0}^{t}\int_{B}\psi_\infty|F(u\cdot\nabla\frac {\tilde{\psi}} {\psi_\infty})|^2 dRdsd\xi
&\leq C\int_{S(t)}d\xi \int_{0}^{t}\int_{B}\|\psi_\infty|F(u\cdot\nabla\frac {\tilde{\psi}} {\psi_\infty})|^2\|_{L^{\infty}}dRds \\ \notag
&\leq C(1+t)^{-1} \int_{0}^{t}\|u\|^2_{L^{2}}\|\nabla\tilde{\psi}\|^2_{L^{2}(\mathcal{L}^{2})}ds \\ \notag
&\leq C(1+t)^{-1}.
\end{align}
A similar argument for the last term in \eqref{ineq21} yields
\begin{align}\label{ineq24}
\int_{S(t)}\int_{0}^{t}\int_{B}\psi_\infty|F(\nabla u\cdot{R}\frac {\tilde{\psi}} {\psi_\infty})|^2 dRdsd\xi
&\leq C\int_{S(t)}d\xi \int_{0}^{t}\int_{B}\|\psi_\infty|F(\nabla u\cdot{R}\frac {\tilde{\psi}} {\psi_\infty})|^2\|_{L^{\infty}}dRds \\ \notag
&\leq C(1+t)^{-1} \int_{0}^{t}\|\nabla u\|^2_{L^{2}}\|\tilde{\psi}\|^2_{L^{2}(\mathcal{L}^{2})}ds \\ \notag
&\leq C(1+t)^{-1}.
\end{align}
Plugging the above estimates into \eqref{ineq21}, we obtain
\begin{align}\label{ineq25}
\int_{S(t)}|\hat{u}(t,\xi)|^2 d\xi\leq C((1+t)^{-1}+(1+t)^{-1} \int_{0}^{t}\|u\|^4_{L^{2}}ds).
\end{align}
According to \eqref{ineq18} and \eqref{ineq25}, we deduce that
\begin{align}\label{ineq26}
&\frac d {dt} (\|u\|^2_{L^2}+\|\tilde{\psi}\|^2_{L^2(\mathcal{L}^2)})+\frac {3} {1+t}\| u\|^2_{L^2}
+\int_{\mathbb{R}^{2}}\int_{B}\psi_{\infty}|\nabla_{R}\frac{\tilde{\psi}}{\psi_{\infty}}|^{2}dRdx  \\ \notag
&\leq C\frac {3} {1+t} ((1+t)^{-1}+(1+t)^{-1} \int_{0}^{t}\|u\|^4_{L^{2}}ds).
\end{align}
Multiplying $(1+t)^2$ to \eqref{ineq26} and applying Lemma \eqref{Lemma1}, then we get
\begin{align*}
&\frac d {dt} ((1+t)^2\|u\|^2_{L^2}+(1+t)^2\|\tilde{\psi}\|^2_{L^2(\mathcal{L}^2)})+3(1+t)\| u\|^2_{L^2}
+(1+t)^2\int_{\mathbb{R}^{2}}\int_{B}\psi_{\infty}|\nabla_{R}\frac{\tilde{\psi}}{\psi_{\infty}}|^{2}dRdx  \\
&\leq 3C+3C\int_{0}^{t}\|u\|^4_{L^{2}}ds+2(1+t)\|u\|^2_{L^2}+2(1+t)\|\tilde{\psi}\|^2_{L^2(\mathcal{L}^2)} \\
&\leq 3C+3C\ep_0\int_{0}^{t}\|u\|^2_{L^{2}}ds+2(1+t)\|u\|^2_{L^2}
+2C(1+t)\int_{\mathbb{R}^{2}}\int_{B}\psi_{\infty}|\nabla_{R}\frac{\tilde{\psi}}{\psi_{\infty}}|^{2}dRdx.
\end{align*}
from which we deduce that
\begin{align*}
&(1+t)^2\|u\|^2_{L^2}+(1+t)^2\|\tilde{\psi}\|^2_{L^2(\mathcal{L}^2)}+\int_{0}^{t}(1+s)\| u\|^2_{L^2}ds  \\
&\leq \|u_0\|^2_{L^2}+\|\tilde{\psi}_0\|^2_{L^2(\mathcal{L}^2)}+Ct+C\ep_0\int_{0}^{t}\int_{0}^{s}\|u\|^2_{L^{2}}ds'ds  \\
&\leq Ct+C\ep_0\int_{0}^{t}(1+s)^{\frac 1 3}+(1+s)\|u\|^2_{L^{2}}ds  \\
&\leq C(1+t)^{\frac 4 3}+C\ep_0\int_{0}^{t}(1+s)\|u\|^2_{L^{2}}ds.
\end{align*}
For sufficiently small $\ep_0$, we obtain the initial $L^2$ decay rate for $u$
\begin{align}\label{ineq27}
\|u\|^2_{L^2}+\|\tilde{\psi}\|^2_{L^2(\mathcal{L}^2)}\leq  C(1+t)^{-\frac 2 3}.
\end{align}
With the above $L^2$ decay estimate for $u$, we can improve the $L^2$ decay rate in \eqref{ineq10} by using the bootstrap argument.
According to \eqref{ineq22}, we have
\begin{align*}
\int_{S(t)}\int_{0}^{t}|F(u\otimes u)|^2 dsd\xi
&\leq C(1+t)^{-1} \int_{0}^{t}\|u\|^4_{L^{2}}ds \\ \notag
&\leq C(1+t)^{-1} \int_{0}^{t}(1+s)^{-\frac 4 3}ds \\ \notag
&\leq C(1+t)^{-1}.
\end{align*}
Then the proof of \eqref{ineq25} implies that
\begin{align}\label{ineq28}
\int_{S(t)}|\hat{u}(t,\xi)|^2 d\xi\leq C(1+t)^{-1}.
\end{align}
According to \eqref{ineq18} and \eqref{ineq28}, we deduce that
\begin{align}\label{ineq29}
&\frac d {dt} (\|u\|^2_{L^2}+\|\tilde{\psi}\|^2_{L^2(\mathcal{L}^2)})+\frac {3} {1+t}\| u\|^2_{L^2}
+\int_{\mathbb{R}^{2}}\int_{B}\psi_{\infty}|\nabla_{R}\frac{\tilde{\psi}}{\psi_{\infty}}|^{2}dRdx\leq \frac {3C} {(1+t)^2},
\end{align}
which implies that
\begin{align}\label{ineq30}
\|u\|_{L^2}+\|\tilde{\psi}\|_{L^2(\mathcal{L}^2)}\leq  C(1+t)^{-\frac 1 2}.
\end{align}
To improve the decay rate for $\|\tilde{\psi}\|_{L^2(\mathcal{L}^2)}$, we need to consider the $L^2$ decay rate for $\nabla u$. \\
We claim that
\begin{align}\label{ineq31}
\|\nabla u\|_{L^2}+\|\nabla\tilde{\psi}\|_{L^2(\mathcal{L}^2)}+\|\tilde{\psi}\|_{L^2(\mathcal{L}^2)}\leq  C(1+t)^{-1}.
\end{align}
Applying the operator $\nabla^{\alpha}$ with $|\alpha|=1$ to \eqref{eq1} and using standard energy estimation, we have
\begin{align*}
&\frac 1 2 \frac d {dt} (\|\nabla^{\alpha}u\|^2_{L^2}+\|\nabla^{\alpha}\tilde{\psi}\|^2_{L^2(\mathcal{L}^2)})+\|\nabla \nabla^{\alpha} u\|^2_{L^2}+\int_{\mathbb{R}^{2}}\int_{B}\psi_{\infty}|\nabla_{R}\frac{\nabla^{\alpha}\tilde{\psi}}{\psi_{\infty}}|^{2}dRdx  \\
&=\int_{\mathbb{R}^{2}}\int_{B}\nabla^{\alpha}\nabla_R\cdot(\nabla u R\tilde{\psi})\frac{\nabla^{\alpha}\tilde{\psi}}{\psi_{\infty}}dRdx
+\int_{\mathbb{R}^{2}}(\nabla^{\alpha}u\cdot\nabla)u\cdot\nabla^{\alpha}udx
+\int_{\mathbb{R}^{2}}\int_{B}(\nabla^{\alpha}u\cdot\nabla)\tilde{\psi}\nabla^{\alpha}\tilde{\psi}dRdx.
\end{align*}
Integrating by part and using Lemma \eqref{Lemma1}, Theorem \eqref{th1}, we get
\begin{align*}
&\int_{\mathbb{R}^{2}}\int_{B}\nabla^{\alpha}\nabla_R\cdot(\nabla u R\tilde{\psi})\frac{\nabla^{\alpha}\tilde{\psi}}{\psi_{\infty}}dRdx
=-\int_{\mathbb{R}^{2}}\int_{B}\nabla^{\alpha}(\nabla u R\tilde{\psi})\nabla_R\frac{\nabla^{\alpha}\tilde{\psi}}{\psi_{\infty}}dRdx \\
&\leq C(\|\nabla \nabla^{\alpha} u\|_{L^2}\|\tilde{\psi}\|_{L^{\infty}(\mathcal{L}^2)})
(\int_{\mathbb{R}^{2}}\int_{B}\psi_{\infty}|\nabla_{R}\frac{\nabla^{\alpha}\tilde{\psi}}{\psi_{\infty}}|^{2}dRdx)^{\frac 1 2} \\
&+\|\nabla u\|_{L^{\infty}}\|\nabla^{\alpha}\tilde{\psi}\|_{L^2(\mathcal{L}^2)})
(\int_{\mathbb{R}^{2}}\int_{B}\psi_{\infty}|\nabla_{R}\frac{\nabla^{\alpha}\tilde{\psi}}{\psi_{\infty}}|^{2}dRdx)^{\frac 1 2}) \\
&\leq C\ep_0 (\|\nabla \nabla^{\alpha} u\|^2_{L^2}+\int_{\mathbb{R}^{2}}\int_{B}\psi_{\infty}|\nabla_{R}\frac{\nabla^{\alpha}\tilde{\psi}}{\psi_{\infty}}|^{2}dRdx).
\end{align*}
Using Lemma \eqref{Lemma0} with $p=3$ and Theorem \eqref{th1}, we have
\begin{align*}
&\int_{\mathbb{R}^{d}}(\nabla^{\alpha}u\cdot\nabla)u\cdot\nabla^{\alpha}udx\leq C\|\nabla u\|^3_{L^3}\leq C \|\nabla u\|^2_{L^2} \|\nabla^2 u\|_{L^2}
\leq C\ep_0 \|\nabla^2 u\|^2_{L^2}.
\end{align*}
Using Lemma \eqref{Lemma0}, Lemma \eqref{Lemma1} and Theorem \eqref{th1}, we get
\begin{align*}
\int_{\mathbb{R}^{2}}\int_{B}(\nabla^{\alpha}u\cdot\nabla)\tilde{\psi}\nabla^{\alpha}\tilde{\psi}dRdx
&\leq \|\nabla\tilde{\psi}\|_{L^2(\mathcal{L}^2)}\|\nabla^{\alpha} u\|_{L^{\infty}}
\|\nabla^{\alpha}\tilde{\psi}\|^2_{L^2(\mathcal{L}^2)}  \\
&\leq C\ep_0 \int_{\mathbb{R}^{2}}\int_{B}\psi_{\infty}|\nabla_{R}\frac{\nabla\tilde{\psi}}{\psi_{\infty}}|^{2}dRdx.
\end{align*}
Summing up all the estimates for $|\alpha|=1$, we deduce that
\begin{align}\label{ineq32}
\frac d {dt} (\|\nabla u\|^2_{L^2}+\|\nabla\tilde{\psi}\|^2_{L^2(\mathcal{L}^2)})+\|\nabla^2 u\|^2_{L^2}
+\int_{\mathbb{R}^{2}}\int_{B}\psi_{\infty}|\nabla_{R}\frac{\nabla\tilde{\psi}}{\psi_{\infty}}|^{2}dRdx\leq 0,
\end{align}
which implies that
\begin{align}\label{ineq33}
\frac d {dt} (\|\nabla u\|^2_{L^2}+\|\nabla\tilde{\psi}\|^2_{L^2(\mathcal{L}^2)})+\frac {3} {1+t}\|\nabla u\|^2_{L^2}
+\int_{\mathbb{R}^{2}}\int_{B}\psi_{\infty}|\nabla_{R}\frac{\nabla\tilde{\psi}}{\psi_{\infty}}|^{2}dRdx\leq \frac {3} {1+t}\int_{S(t)}|\xi|^2 |\hat{u}(\xi)|^2 d\xi.
\end{align}
According to \eqref{ineq30}, we have
\begin{align*}
\frac {3} {1+t}\int_{S(t)}|\xi|^2 |\hat{u}(\xi)|^2 d\xi\leq 9 (1+t)^{-2}\|u\|^2_{L^2}\leq C (1+t)^{-3}.
\end{align*}
This together with \eqref{ineq33} ensures that
\begin{align}\label{ineq34}
\|\nabla u\|_{L^2}+\|\nabla\tilde{\psi}\|_{L^2(\mathcal{L}^2)}\leq C (1+t)^{-1}.
\end{align}
Using standard energy estimation to the second equation in \eqref{eq1}, we get
\begin{align*}
&\frac d {dt} \|\tilde{\psi}\|^2_{L^2(\mathcal{L}^2)}+2\int_{\mathbb{R}^{2}}\int_{B}\psi_{\infty}|\nabla_{R}\frac{\tilde{\psi}}{\psi_{\infty}}|^{2}dRdx  \\
&=2\int_{\mathbb{R}^{2}}\int_{B}\nabla_R\cdot(\nabla u R\tilde{\psi})\frac{\tilde{\psi}}{\psi_{\infty}}dRdx+2\int_{\mathbb{R}^{2}}\int_{B}\nabla u\cdot{R}\nabla_{R}\mathcal{U}\tilde{\psi}dRdx.
\end{align*}
Using Lemma \eqref{Lemma1}, Lemma \eqref{Lemma2} and Theorem \eqref{th1}, we have
\begin{align*}
\int_{\mathbb{R}^{2}}\int_{B}\nabla_R\cdot(\nabla u R\tilde{\psi})\frac{\tilde{\psi}}{\psi_{\infty}}dRdx
=\int_{\mathbb{R}^{2}}\int_{B}\psi_{\infty}\nabla u R\frac{\tilde{\psi}}{\psi_{\infty}}
\nabla_R\frac{\tilde{\psi}}{\psi_{\infty}}dRdx\leq C\ep_0 \int_{\mathbb{R}^{2}}\int_{B}\psi_{\infty}|\nabla_{R}\frac{\tilde{\psi}}{\psi_{\infty}}|^{2}dRdx
\end{align*}
and
\begin{align*}
\int_{\mathbb{R}^{2}}\int_{B}\nabla u\cdot{R}\nabla_{R}\mathcal{U}\tilde{\psi}dRdx
=\int_{\mathbb{R}^{2}}\nabla u\cdot\tau dx\leq \delta\int_{\mathbb{R}^{2}}\int_{B}\psi_{\infty}|\nabla_{R}\frac{\tilde{\psi}}{\psi_{\infty}}|^{2}dRdx+C_\delta \|\nabla u\|^2_{L^2}.
\end{align*}
Using Lemma \eqref{Lemma1}, we obtain
\begin{align*}
\frac d {dt} \|\tilde{\psi}\|^2_{L^2(\mathcal{L}^2)}+\|\tilde{\psi}\|^2_{L^2(\mathcal{L}^2)}\leq C_\delta \|\nabla u\|^2_{L^2}.
\end{align*}
This together with \eqref{ineq34} implies that
\begin{align*}
\|\tilde{\psi}\|^2_{L^2(\mathcal{L}^2)}
&\leq \|\tilde{\psi}_0\|^2_{L^2(\mathcal{L}^2)}e^{-t}+C\int_{0}^{t}e^{-(t-s)}\|\nabla u\|^2_{L^2}ds  \\
&\leq C(e^{-t}+\int_{0}^{t}e^{-(t-s)}(1+s)^{-2}ds)  \\
&\leq C(1+t)^{-2},
\end{align*}
which means that we prove the claim \eqref{ineq31}.
We thus complete the proof of Proposition \eqref{pro2}.
\end{proof}

By \eqref{ineq31}, Lemma \eqref{Lemma0} and Proposition \eqref{pro2}, we obtain the following corollary.
\begin{coro}
Under the assumption of Theorem \eqref{th2}, for $2\leq p<\infty$, we have
\begin{align}
\|u\|_{L^p}\leq C(1+t)^{-(1-\frac 1 p)},~~
\|\tilde{\psi}\|_{L^p(\mathcal{L}^2)}\leq C(1+t)^{-1}.
\end{align}
\end{coro}

\smallskip
\noindent\textbf{Acknowledgments} This work was
partially supported by the National Natural Science Foundation of China (No.11671407 and No.11701586), the Macao Science and Technology Development Fund (No. 098/2013/A3), and Guangdong Province of China Special Support Program (No. 8-2015),
and the key project of the Natural Science Foundation of Guangdong province (No. 2016A030311004).


\phantomsection
\addcontentsline{toc}{section}{\refname}
\bibliographystyle{abbrv} 
\bibliography{Feneref}

\end{document}